\documentclass[a4paper,12pt,reqno]{amsart}
\usepackage[T1]{fontenc}
\usepackage[utf8]{inputenc}
\usepackage[libertinus,vvarbb]{newtx}
\usepackage[margin=1in]{geometry}
\usepackage{enumitem}
\usepackage{graphicx,color}
\usepackage{xurl}
\usepackage[bookmarksdepth=2]{hyperref}
\usepackage{booktabs}
\usepackage{multirow}

\usepackage[
backend=biber,
style=alphabetic,
sortcites=true,
sorting=nyt,
maxbibnames=99, 
]{biblatex}
\numberwithin{equation}{section}

\newtheorem{theorem}{Theorem}[section]

\newtheorem{lemma}[theorem]{Lemma}
\newtheorem{corollary}[theorem]{Corollary}

\theoremstyle{definition}
\newtheorem*{defn}{Definition}

\theoremstyle{remark}

\DeclareMathOperator{\curv}{H}

\newcommand{\ind}{\chi}

\newcommand{\R}{\mathbb{R}}
\newcommand{\Z}{\mathbb{Z}}

\renewcommand{\rho}{\varrho}
\renewcommand{\epsilon}{\varepsilon}
\renewcommand{\theta}{\vartheta}

\renewcommand{\le}{\leqslant}
\renewcommand{\ge}{\geqslant}

\renewcommand{\geq}{\ge}
\renewcommand{\leq}{\le}
\begin{document}

\title[Volumetric density estimates for nonlocal minimal surfaces]{Volumetric density estimates for nonlocal minimal surfaces}

\author{Mateusz Kwaśnicki}
\thanks{Work supported by the National Science Centre, Poland, grant no.\@ 2023/49/B/ST1/04303}
\address{Mateusz Kwaśnicki \\ \textnormal{Department of Analysis and Stochastic Processes \\ Wrocław University of Science and Technology \\ Wybrzeże Wyspiańskiego 27 \\ 50-370 Wrocław, Poland}}
\email{\href{mailto:mateusz.kwasnicki@pwr.edu.pl}{\textsf{mateusz.kwasnicki@pwr.edu.pl}}}

\author[J. Thompson]{Jack Thompson}
\address{Jack Thompson \\ \textnormal{Department of Mathematics \\ KTH \\ Lindstedsv\"agen 25 \\ 100 44 Stockholm, Sweden}}
\email{\href{mailto:jackth@kth.se}{\textsf{jackth@kth.se}}}
\date{\today}

\begin{abstract}
In this article, we prove that viscosity subsolutions to nonlocal mean curvature-type equations satisfy universal volumetric estimates at all scales provided their boundary has zero Lebesgue measure. Our results hold for a general class of symmetric kernels including kernels that are comparable to the fractional Laplacian. Furthermore, we prove that subsolutions with low density (with respect to a universal constant) necessarily have `fat boundary', that is, have topological boundary with positive Lebesgue measure. 

\end{abstract}

\maketitle
\thispagestyle{empty}

%
%
\section{Introduction}

Minimal surfaces, famously arising as a model for soap films, appear ubiquitously in the applied sciences including in chemistry, materials science, biology, general relativity, and even architectural design. Their analysis is a rich and challenging topic, and has required the development of many techniques that are now widely used across geometry, topology, measure theory, and the analysis of partial differential equations. Of particular significance is the regularity theory for minimal surfaces which has seen significant progress over the past century and is still currently an active area of research. Some references on the foundational theory of minimal surfaces include \cite{maggi_sets_2012,colding_course_2011,white_lectures_2016}.

More recently, over the last decade and a half, there has been a consolidated interest in nonlocal or fractional minimal surfaces which may be viewed as a non-infinitesimal generalisation of the classical minimal surfaces. In applications, nonlocal minimal surfaces arise naturally in phase transition models where the underlying diffusion can exhibit long-range jumps, see \cite{dipierro_perspectives_2023}. For general surveys on nonlocal minimal surfaces and their regularity theory, see \cite{dipierro_nonlocal_2018,cinti_regularity_2022,serra_nonlocal_2024}.

The study of nonlocal minimal surfaces began with the celebrated paper of Caffarelli, Roquejoffre, and Savin \cite{caffarelli_nonlocal_2010} where they introduced so-called \(s\)-minimal surfaces, the prototypical example of a nonlocal minimal surface. Concretely, if \(s\in (0,1)\) then the \emph{\(s\)-perimeter} of a measurable set \(E\subset \R^d\) in a bounded open set \(\Omega\subset \R^d\) is defined as \begin{equation}
    \operatorname{Per}_s(E;\Omega) = \frac 1 4  \iint_{ Q(\Omega)}  \frac{\lvert \tilde \chi_E(x)-\tilde \chi_E(y)  \rvert}{\lvert x-y\rvert^{d+s}}  \, dy \, dx . \label{rKeVDPTi}
\end{equation} Here  \begin{equation}
     Q(\Omega) = \R^{2d}\setminus \bigl ( (\R^d\setminus  \Omega) \times (\R^d\setminus  \Omega) \bigr ) = \bigl ( \Omega \times \Omega \bigr ) \cup \bigl ( (\R^d\setminus \Omega) \times \Omega \bigr ) \cup \bigl ( \Omega \times (\R^d\setminus \Omega) \bigr ) \label{eq:defn-Q}
\end{equation} and \(\tilde \chi_E=\chi_{\R^d\setminus E} - \chi_E\) where \(\chi_F\) denotes the characteristic function of a Borel set \(F\). Furthermore, a minimiser~\(E\) of the \(s\)-perimeter satisfies the equation\begin{equation}
    \curv_{s,E} = 0 \quad \text{on } \partial E, \label{ZC9WGPJu}
\end{equation}
where
\begin{equation}
\curv_{s, E}(x) =\lim_{\varepsilon \rightarrow 0^+} \int_{\R^d \setminus B_\varepsilon(x)} \frac{\tilde \chi_E(y)}{\lvert x - y\rvert^{d + s}} \,d y, \qquad x\in \partial E \label{SswSVJj9}
\end{equation} provided this limit exists. In analogy with the theory of classical minimal surfaces, the object \(\curv_{s, E}\) is often called the \emph{\(s\)-mean curvature} and solutions to~\eqref{ZC9WGPJu} are called \emph{\(s\)-minimal surfaces}\footnote{In \cite{caffarelli_nonlocal_2010}, \(s\)-minimal surfaces are defined as \emph{minimisers} of the \(s\)-perimeter. In this article, however, an \(s\)-minimal surface always refers to a stationary solution of the \(s\)-perimeter (equivalently a solution to~\eqref{ZC9WGPJu} for sufficiently regular sets) in line with the terminology used in classical differential geometry. We also adopt this terminology for more general kernels.}. More broadly, we use the terms `nonlocal perimeter', `nonlocal mean curvature', and `nonlocal minimal surface' to refer to analogues of~\eqref{rKeVDPTi},~\eqref{ZC9WGPJu}, and~\eqref{SswSVJj9} respectively with more general kernels, see~\S\ref{JG0TklX4} for details.

\subsection{Density estimates}

In the regularity theory of classical and nonlocal minimal surfaces, a central role is played by density estimates which are estimates describing quantitatively how `close' portions of a surface can be to one another. The term `density estimate' can refer to one of several different, but related inequalities. Indeed, given an open set \(E\subset \R^d\) with \(0\in \partial E\), a \emph{volumetric density estimate} refers to the inequality \begin{equation}
    \min \{ \lvert E \cap B_\rho \rvert , \lvert B_\rho \setminus E \rvert \} \ge C\rho^d  , \qquad \rho >0 \label{VDE} \tag{V}
\end{equation} for some \(C >0\). Moreover, a \emph{surface density estimate from above} (also called a \emph{perimetric density estimate from above}) refers to the inequality \begin{equation}
    \operatorname{Per} (E ; B_\rho ) \le C \rho^{d-1}  , \qquad \rho >0 \label{USDE} \tag{SA}
\end{equation} and a \emph{surface density estimate from below} refers to the inequality \begin{equation}
    \operatorname{Per} (E ; B_\rho ) \ge C \rho^{d-1}  , \qquad \rho >0 \label{LSDE} \tag{SB}
\end{equation} for some \(C>0\). Of course, if \(E\) is sufficiently smooth then~\eqref{VDE},~\eqref{USDE}, and~\eqref{LSDE} all hold for \(\rho\) sufficiently small and \(C\) depending on the regularity of \(E\), so the difficulty lies in establishing these inequalities for all \(\rho>0\) and for a universal constant \(C>0\).

\subsubsection{Classical minimal surfaces}
If \(E\subset \R^d\), \(d\ge 2\), is a set of locally finite perimeter that is locally minimising with respect to the classical perimeter, it is known that~\eqref{VDE},~\eqref{USDE}, and~\eqref{LSDE} hold for all \(\rho>0\) and \(C=C(d)\), see, for example, \cite[Theorem 16.14]{maggi_sets_2012}. Indeed,~\eqref{USDE} follows from the inequality \begin{equation}
    \operatorname{Per}(E; B_\rho) \le \operatorname{Per}(B_\rho ; E) \label{AbmqKRv5}
\end{equation} which holds due to the minimality of \(E\). Furthermore, one can establish~\eqref{VDE} using~\eqref{AbmqKRv5} in conjunction with the co-area formula and the isoperimetric inequality, and~\eqref{VDE} implies~\eqref{LSDE} via the relative isoperimetric inequality: \begin{equation}
     \operatorname{Per} (E ; B_\rho ) \geq C(d) \bigl (  \min \{ \lvert E \cap B_\rho \rvert , \lvert B_\rho \setminus E \rvert \} \bigr ) ^{\frac{d-1}d }. \label{yzR232ke}
\end{equation}  

When \(\partial E\) is only assumed to be a minimal surface, that is \(E\) is stationary with respect to the perimeter, then the validity of~\eqref{VDE},~\eqref{USDE}, and~\eqref{LSDE} changes. We still have that~\eqref{LSDE} is valid for minimal surfaces since \begin{equation*}
   \Phi_E: \rho \mapsto \rho^{-(d-1)}\operatorname{Per} ( E; B_\rho)  
\end{equation*} is monotone increasing and \(\lim_{\rho \to  0^+} \Phi_E(\rho) \geq C(d)>0\), see \cite[Theorem 17.16, Corollary 17.18]{maggi_sets_2012}. However,~\eqref{VDE} and~\eqref{USDE} are no longer true for all \(\rho>0\) and a universal \(C\). An easy counter-example for~\eqref{VDE} (with disconnected boundary) is given by the slab \( E = (0,1) \times \R^{d-1} \).  Alternatively, a counter-example with connected boundary is given by \(E\) such that \(\partial E\) is the catenoid when \(d=3\) or its higher dimensional analogue when \(d\geq 4\). A counter-example for~\eqref{USDE} is given by the periodic slab  \begin{equation}
    E_\delta =\R^{d-1} \times \bigcup_{k\in \Z} \bigl (2k\delta ,(2k+1)\delta \bigr ), \qquad \delta>0 \label{ySObyBJy}
\end{equation} since it has arbitrarily high perimeter in \(B_1\) provided \(\delta\) is sufficiently small. 

For stable or finite Morse index minimal surfaces, the validity of~\eqref{VDE},~\eqref{USDE}, and~\eqref{LSDE} remains the same as in the stationary case. Indeed, both the slab and the periodic slab are stable minimal surfaces, and the catenoid has finite Morse index, so they are still counter-examples to~\eqref{VDE} and~\eqref{USDE} under these additional assumptions. As far as we are aware, the validity of~\eqref{USDE} for stable, \emph{connected} minimal surfaces is unknown. We conclude our discussion on density estimates for classical minimal surfaces by mentioning that for two-dimensional simply connected, immersed, stable minimal surfaces, an analogous inequality to~\eqref{USDE}-\eqref{LSDE} holds for geodesic balls. More precisely, let \(M^2\subset \R^3\) be a simply connected, minimal immersed surface. Then if \( \mathcal B_\rho \) is an open geodesic ball of radius \(\rho\) centred at \(0\in M\) such that \(\mathcal B_\rho \cap \partial M=\varnothing\) and \(\mathcal B_\rho\) is stable, then \begin{equation}
    \pi \rho^2 \leq \mathcal H^2 ( \mathcal B_\rho ) \leq \frac 43 \pi \rho^2 , \label{GAcMkXQd}
\end{equation} see \cite[Corollary 2.2]{colding_estimates_2002} and \cite{pogorelov_stability_1981}. See also \cite[Theorem 3.4]{meeks_proofs_2005} and \cite[Lemma 34]{white_lectures_2016}. The validity of~\eqref{GAcMkXQd} in higher dimensions (with different constants) is currently unknown.

\subsubsection{Nonlocal minimal surfaces}
Now, we turn our attention to density estimates for nonlocal minimal surfaces. First, we will discuss the results for \(s\)-minimal surfaces, then we will mention more general kernels. For minimisers of the \(s\)-perimeter,~\eqref{VDE},~\eqref{USDE}, and~\eqref{LSDE} all hold for a universal \(C\) and for all \(\rho>0\). Indeed, for minimisers of the \(s\)-perimeter,~\eqref{VDE} is proved in the original paper of Caffarelli, Roquejoffre, and Savin \cite[Theorem 4.1]{caffarelli_nonlocal_2010} following a strategy inspired by the classical case. Furthermore, this implies~\eqref{LSDE} via~\eqref{yzR232ke} as in the classical case. 

For~\eqref{USDE}, this was first established in \cite[Theorem 1.7]{cinti_quantitative_2019} for \(s\)-minimal surfaces that are stable under outer rearrangements\footnote{In this article, for both classical and nonlocal perimeters, a stable set is a stationary set with non-negative second variation with respect to inner variations, for example, as in \cite[Chapter 17]{maggi_sets_2012}. In \cite{cinti_quantitative_2019}, an alternative ad hoc definition of stability (referred to here as stable under outer rearrangements) was used. This notion of stability is equivalent to the usual notion of stability for \(C^2\) sets, but a stronger notion for less regular sets. For a general discussion regarding these two definitions see \S2 in \cite{cabre_stable_2020} and \S1.2 in \cite{caselli_stable_2025}.\label{YpIoIIl7}}. We emphasise that this is a surprising result since it is a control on the classical perimeter which is stronger than a control on the \(s\)-perimeter, the energy for \(s\)-minimal surfaces. The result of \cite{cinti_quantitative_2019} was later extended to the case of \(s\)-minimal surfaces with finite Morse index in \cite[Theorem 1.26]{caselli2025yausconjecturenonlocalminimal}. Furthermore,~\eqref{USDE} does not hold for general \(s\)-minimal surfaces with~\eqref{ySObyBJy} providing a counter-example. 

For~\eqref{LSDE}, this holds for all \(s\)-minimal surfaces which was proved in \cite{thompson_density_2026}. The proof relies on the monotonicity formula for \(s\)-minimal surfaces, see \cite[Theorem 3.4]{caselli_fractional_2024}, and an interpolation inequality between the volume, the perimeter, and the \(s\)-perimeter. 

The only remaining inequality is the volumetric estimate~\eqref{VDE} for non-minimising sets. In this case, it is known that~\eqref{VDE} holds for sets satisfying~\eqref{USDE}, for example stable/finite Morse index sets, see \cite[Proposition 2.9]{cozzi_halfspace_inprep}. Furthermore,~\eqref{VDE} is known to hold at small scales if \(0\in \partial E\) can be touched from inside or outside by a ball, see \cite[Proposition 2.10]{cozzi_halfspace_inprep}. The main result of the current article, in its most simple form, establishes~\eqref{VDE} for all sufficiently regular \(s\)-minimal surfaces: \begin{theorem} \label{TooFvLFR}
    Let \(s\in (0,1)\), \(\alpha\in (s,1)\), \(E\subset \R^d \) be an open set with \(C^{1,\alpha}\) boundary, and \(0\in \partial E\). If\, \(\curv_{s,E} =0\) on \(\partial E\) then \begin{equation*}
      \lvert E \cap B_R \rvert \geq C \lvert B_R\rvert \qquad \text{for all }R>0.
    \end{equation*} The constant \(C>0\) depends only on \(d\) and \(s\).
\end{theorem}

For nonlocal minimal surfaces with respect to more general kernels, there are very few results, and they are primarily focused on~\eqref{USDE}. Indeed, in~\cite{cinti_quantitative_2019} they deal directly with a very broad class of kernel and establish~\eqref{USDE} for nonlocal minimal surfaces that are stable under outer rearrangements, see Footnote~\ref{YpIoIIl7}. Furthermore, in \cite{caselli2025yausconjecturenonlocalminimal,caselli_fractional_2024}, they establish~\eqref{USDE} for \(s\)-minimal sets with finite Morse index on a closed Riemannian manifold \(M\), see \cite[Theorem 1.26]{caselli2025yausconjecturenonlocalminimal}. In this setting, the homogeneous kernel in~\eqref{SswSVJj9} is replaced with the natural kernel on \(M\), defined in terms of the heat kernel, see \S1.3 in \cite{caselli2025yausconjecturenonlocalminimal}. We also mention that for kernels satisfying~\eqref{wqeJL6Xs}-\eqref{M4Xlk3oo}, the periodic slab~\eqref{ySObyBJy} demonstrates~\eqref{USDE} cannot hold without imposing further assumptions such as stability/finite Morse index.

This brings us to the validity of~\eqref{VDE} and~\eqref{LSDE} for general nonlocal minimal surfaces. In this article, we also establish~\eqref{VDE} for nonlocal minimal surfaces for a general class of symmetric, uniformly elliptic kernels, see Corollary~\ref{k9pGd2yU}. This immediately implies, via the relative isoperimetric inequality~\eqref{yzR232ke} that~\eqref{LSDE} holds for the same class of kernels, see Corollary~\ref{IP8ScC2r}. As we mentioned earlier,~\eqref{LSDE} was already established in \cite{thompson_density_2026} for \(s\)-minimal surfaces, but the proof relied on the monotonicity which is not available for general kernels.

\subsection{Main results} \label{JG0TklX4}
 To state our results precisely, let \(d\geq 1\), \(k:(0,+\infty) \to (0,+\infty)\), and \(K : \R^d \setminus \{0\} \to \R\) be Borel functions such that \begin{equation}
   \lambda k(\vert y \vert ) \leq K(y) \leq \Lambda k(\vert y \vert ) \text { with } 0<\lambda\leq \Lambda \label{wqeJL6Xs}
\end{equation} and \begin{equation}
    K(y)=K(-y) . \label{M4Xlk3oo}
\end{equation} Furthermore, we assume that:  \begin{align}
&t \mapsto t^{d-1} k(t) \in L^1((\bar r,+\infty)), \label{HKZolO8l}\\ 
    &k \text{ is non-increasing on }(0,\bar r), \label{BcfVd3vx}
\end{align} and  \begin{align}
    A \bigg ( \frac{r_2}{r_1} \bigg )^a \leq \frac{k(r_2)}{k(r_1)} \leq B \bigg ( \frac{r_2}{r_1} \bigg )^b \text{ for all } 0<r_1<r_2\leq \bar r \label{y8vGgWhv}
\end{align} for some constants \(\bar r>0\), \(0<A\leq B <+\infty\), and \(-d-1<a<b<-d\). Assumption~\eqref{y8vGgWhv} can be phrased in words as \(k\) is \(O\)-regularly varying at 0 with Matuszewska indices between \(-d-1\) and \(-d\). For more information on \(O\)-regularly varying functions, one can consult \cite{bingham_regular_1987} or \cite[Appendix A]{grzywny_potential_2018}. Moreover, we assume~\eqref{BcfVd3vx} for convenience, but if \(k\) satisfies~\eqref{y8vGgWhv} then one can always replace \(k\) by a comparable non-increasing function. Throughout this article, whenever we say a quantity depends only on the coefficients of \(k\), we mean that the quantity depends only on \(A, B, a\), and \(b\), but \emph{not} \(\bar r\). We also define the dimensionless quantity \begin{align*}
    T_k =\frac 1 {\bar r ^d k(\bar r )} \int_{\bar r}^{+\infty} t^{d-1}k(t) \, dt 
\end{align*} which appears in some of the proofs.

Then, given a Borel set \(E\subset \R^d\) and an open bounded set \(\Omega\), the nonlocal perimeter associated with the kernel \(K\) in \(\Omega\) is given by \begin{equation*}
    \operatorname{Per}_K(E;\Omega) = \frac 1 4  \iint_{ Q(\Omega)}  \lvert \tilde \chi_E(x)-\tilde \chi_E(y)  \rvert \, K(x-y) \, dy \, dx 
\end{equation*} with \(Q(\Omega)\) defined by~\eqref{eq:defn-Q}, and \(\tilde \chi_E=\chi_{\R^d\setminus E} - \chi_E\) where \(\chi_F\) denotes the characteristic function of a Borel set \(F\). Furthermore, the `full' nonlocal perimeter associated with \(K\) is \begin{equation*}
    \operatorname{Per}_K(E)=\operatorname{Per}_K(E;\R^d) = \int_{\R^d\setminus E} \int_E K(x-y) \,d y\,d x.
\end{equation*} In the particular case \(K(y)=\lvert y \rvert^{-(d+s)}\), we refer to \(\operatorname{Per}_K\) as the \(s\)-perimeter and write \(\operatorname{Per}_K=\operatorname{Per}_s\) as in~\eqref{rKeVDPTi}.

Next, we define the nonlocal mean curvature associated with the kernel \(K\) by \begin{equation}
    \curv_{K,E}(x) = \lim_{\varepsilon \to 0^+} \int_{\R^d \setminus B_\varepsilon(x)} \tilde \chi_E (y) K(x-y) \,d y \label{DVf4VFSd} , \qquad x\in \partial E
\end{equation} provided this limit exists. This definition is motivated by the fact that for \(E\) sufficiently regular, and \(K\) sufficiently regular and uniformly elliptic, we have that the limit in~\eqref{DVf4VFSd} exists for all \(x\in \partial E\) and the Euler-Lagrange equation for \(\operatorname{Per}_K(E)\) is given by \begin{equation}
    \curv_{K,E} = 0 \qquad \text{on }\partial E, \label{seCTuwK3}
\end{equation} in analogy to classical minimal surfaces, see \cite[Proposition 6.1, Proposition 6.3]{figalli_isoperimetry_2015}. Furthermore, a sufficiently regular set \(E\) is referred to as a nonlocal minimal surface (with respect to \(K\)) if it satisfies~\eqref{seCTuwK3} in the pointwise sense. As with the nonlocal perimeter, in the particular case \(K(y)=\lvert y \rvert^{-(d+s)}\), we refer to \(\curv_{K,E}\) as the \(s\)-mean curvature, and solutions to~\eqref{seCTuwK3} as \(s\)-minimal surfaces.

For rough sets,~\eqref{seCTuwK3} can no longer be understood in the pointwise sense, so it must be reinterpreted in a weak formulation. One such weak formulation is the notion of viscosity solutions to nonlocal mean curvature equations, and this is the definition of solution we adopt in this article. We only require the definition of a viscosity subsolution, which we give below:

\begin{defn}[Viscosity subsolutions]
    Let \(\Omega\subset \R^d\) be an open set, \(E\subset \R^d\) be a Borel set\footnote{We do not identify sets up to Lebesgue measure zero, which is a common convention in the literature.}, and \(C_0\in \R\). We say that \(\curv_{K,E} \leq C_0\) in \(\Omega\) in the viscosity sense if for all \(x\in \partial E \cap \Omega\) with an exterior touching ball (i.e. there exists a ball \(B \subset \R^d\setminus E\) with \(x\in \partial B\)) we have \(\curv_{K,E}(x) \leq C_0\).
\end{defn}

Note that the notion of viscosity solution and pointwise solution coincide for sufficiently regular sets. Moreover, for the above definition to make sense, one is required to know \emph{a priori} that~\(\curv_{K, E}(x)\) is well-defined in the principal value sense (albeit possibly infinite) whenever~\(E\) is touched by an exterior ball. We address this and other technical details in Section~\ref{vLKCoqef}. Then our first main result establishes volumetric density estimates for viscosity subsolutions of nonlocal mean curvature equations.

\begin{theorem} \label{Waso0CUq}
    Let \(M\geq 0\), \(K\) satisfy~\eqref{wqeJL6Xs}--\eqref{M4Xlk3oo}, and \(E\subset \R^d \) be a Borel set such that \(\partial E\) has zero Lebesgue measure and \(0\in \partial E\). If \(0<R<\bar r\) and\, \(\curv_{K,E} \leq M R^dk(R)\) in \(B_{R/2}\) in the viscosity sense, then \begin{equation*}
      \lvert E \cap B_R \rvert \geq \delta \lvert B_R\rvert .
    \end{equation*} The constant \(\delta>0\) depends only on \(d\), \(\lambda\), \(\Lambda\), \(M\), the coefficients of \(k\), and \(T_k\).
\end{theorem}

Here and throughout this article, \(\lvert F\rvert \) always refers to the Lebesgue measure of \(F\). We emphasise that the notion of viscosity solutions makes sense for all Borel sets and so solutions could, \emph{a priori}, be very irregular. Indeed, there is no reason that a viscosity subsolution could not have boundary with positive Lebesgue measure, a situation we must explicitly rule out (via our assumptions) in Theorem~\ref{Waso0CUq}. In fact, the assumption in Theorem~\ref{Waso0CUq} that a subsolution has boundary with zero Lebesgue plays a critical role in the proof and we explore this peculiarity further in Theorem~\ref{thm:main} below. Next, we state an easy corollary of Theorem~\ref{Waso0CUq} which may be more readily applied in future articles.

\begin{corollary} \label{ED1qK30X}
    Let \(M\geq 0\), \(K\) satisfy~\eqref{wqeJL6Xs}--\eqref{M4Xlk3oo}, and \(E\subset \R^d \) be a Borel set such that \(\partial E\) has zero Lebesgue measure and \(0\in \partial E\). If \(0<R<\bar r\) and\, \(\curv_{K,E} \leq M R^d k(R)\) in \(B_R\) in the viscosity sense, then \begin{equation*}
      \lvert E \cap B_r \rvert \geq \delta \lvert B_r\rvert  \qquad \text{for all }r\in (0,\min \{ 2R, \bar r\} ).
    \end{equation*} The constant \(\delta>0\) depends only on \(d\), \(\lambda\), \(\Lambda\), \(M\), the coefficients of \(k\), and \(T_k\).
\end{corollary}

Since homogeneous kernels satisfy Assumptions~\eqref{HKZolO8l}-\eqref{y8vGgWhv} for all \(\bar r>0\) and \(T_k\) is dimensionless, Theorem~\ref{Waso0CUq} also implies that if \(K\) is comparable to a homogeneous kernel then subsolutions in the entire space satisfy universal volumetric density estimates at all scales, which we state precisely in the following corollary. 

\begin{corollary} \label{k9pGd2yU}
    Let \(K\) satisfy \(K(y) = K(-y)\) and \begin{align*}
        \frac \lambda {\vert y \vert^{d+s}} \leq K(y) \leq \frac \Lambda {\vert y \vert^{d+s}} \qquad \text{for all }y\in \R^d\setminus \{0\}
    \end{align*} for some \(0<\lambda \leq \Lambda <+\infty\) and \(s\in (0,1)\). Moreover, let \(E\subset \R^d \) be a Borel set such that \(\partial E\) has zero Lebesgue measure and \(0\in \partial E\). If\, \(\curv_{K,E} \leq 0\) in \(\R^d\) in the viscosity sense then \begin{equation*}
      \lvert E \cap B_r \rvert \geq \delta \lvert B_r\rvert  \qquad \text{for all }r>0 .
    \end{equation*} The constant \(\delta>0\) depends only on \(d\), \(s\), \(\lambda\), and \(\Lambda\).
\end{corollary}

Since the notion of viscosity solution and~\eqref{ZC9WGPJu} are equivalent for sets with \(C^{1,\alpha}\) boundary with \(\alpha \in (s,1)\), Corollary~\ref{k9pGd2yU} clearly implies Theorem~\ref{TooFvLFR}. In our final corollary of Theorem~\ref{Waso0CUq}, we show that general nonlocal minimal surfaces satisfy universal fractional surface density estimates from below. In particular, it establishes the validity of~\eqref{LSDE}.

\begin{corollary} \label{IP8ScC2r}
    Let\footnote{Here we use the convention \(\operatorname{Per}_1(E; B_r):=\operatorname{Per} (E;B_r)\) where \(\operatorname{Per} (\cdot)\) is the usual measure-theoretic perimeter functional.} \(\alpha\in (0,1]\), \(K\) satisfy the assumptions of Corollary~\ref{k9pGd2yU}, and \(E\subset \R^d \) be a Borel set such that \(\partial E\) has zero Lebesgue measure and \(0\in \partial E\). If\, \(\curv_{K,E} = 0\) in \(\R^d\) in the viscosity sense then \begin{equation*}
      \operatorname{Per}_\alpha (E;B_r) \geq \delta \operatorname{Per}_\alpha ( B_r)  \qquad \text{for all }r>0 .
    \end{equation*} The constant \(\delta>0\) depends only on \(d\), \(\alpha\), \(s\), \(\lambda\), and \(\Lambda\).
\end{corollary}

We mentioned above that the assumption that the boundary of the subsolution has zero Lebesgue measure plays an important role in the proof of Theorem~\ref{Waso0CUq}. This is further evidenced in Theorem~\ref{thm:main} below where we prove that if a subsolution is sufficiently sparse then the sets of points in the boundary at which the nonlocal mean curvature exists (in the sense of~\eqref{DVf4VFSd}) and equals \(+\infty\) is a set of positive Lebesgue measure:

\begin{theorem}
\label{thm:main}
 Let \(M\geq 0\), \(K\) satisfy~\eqref{wqeJL6Xs}-\eqref{M4Xlk3oo}, and \(E\subset \R^d \) be a Borel set such that \(0\in \partial E\). There exists \(\delta>0\) such that if \,\(\curv_{K,E} \leq M R^dk(R)\) in \(B_{R/2}\) in the viscosity sense with \(0<R<\bar r \) and \begin{equation*}
        \lvert E \cap B_{R} \rvert < \delta  \lvert B_R\rvert 
    \end{equation*} then the set\begin{equation}
        \{ x \in \partial E \cap B_{R/4}  \text{ s.t. } \curv_{K,E}(x) = +\infty \} \label{botfPHsL}
    \end{equation} has positive Lebesgue measure. The constant \(\delta>0\) depends only on \(d\), \(\lambda\), \(\Lambda\), \(M\), the coefficients of \(k\), and \(T_k\).
\end{theorem}

Note carefully that it is not contradictory to have a set \(E\) satisfying an upper bound on its nonlocal curvature in \(B_{R/2}\) in the viscosity sense, but also have a non-empty set of points with infinite nonlocal curvature that is contained in \(B_{R/2}\) since the definition of viscosity solution does not see points that do not have an exterior touching ball. Hence, necessarily, almost every point in~\eqref{botfPHsL} does not have an exterior touching ball. Any open and dense subset of \(B_R \setminus \{0\}\) with Lebesgue measure less than \(\delta \lvert B_R \rvert\) can serve as an example, as the definition of viscosity solution in \(B_{R/2}\) is void in this case.


\subsection{Organisation of paper}
In Section~\ref{vLKCoqef}, we give the definition of viscosity subsolutions for nonlocal curvature equations and establish some simple technical facts showing that this definition is well-defined. In Section~\ref{BQhFLLRN}, we give the proof of Theorem~\ref{Waso0CUq}, Corollary~\ref{ED1qK30X}, and Corollary~\ref{IP8ScC2r}. In Section~\ref{6yuqjipW}, we give the proof of Theorem~\ref{thm:main}.

\section{Preliminaries on viscosity solutions for general kernels} \label{vLKCoqef}

We begin this section by recalling the definition of viscosity (sub/super)solutions for nonlocal curvature equations associated with kernels satisfying~\eqref{wqeJL6Xs}--\eqref{M4Xlk3oo}. 

\begin{defn}[Viscosity solutions] 
    Let \(\Omega\subset \R^d\) be an open set, \(E\subset \R^d\) be a Borel set, and \(C_0\in \R\). We say that: \begin{itemize}
        \item \(\curv_{K,E} \leq C_0\) in \(\Omega\) in the viscosity sense if for all \(x\in \partial E \cap \Omega\) such that there exists a ball \(B \subset \R^d\setminus E\) with \(x\in \partial B\) we have \(\curv_{K,E}(x) \leq C_0\);
        \item \(\curv_{K,E} \geq C_0\) in \(\Omega\) in the viscosity sense if for all \(x\in \partial E \cap \Omega\) such that there exists a ball \(B \subset E\) with \(x\in \partial B\) we have \(\curv_{K,E}(x) \geq C_0\);
        \item \(\curv_{K,E} = C_0\) in \(\Omega\) in the viscosity sense if \(\curv_{K,E} \leq C_0\) in \(\Omega\) in the viscosity sense and \(\curv_{K,E} \geq C_0\) in \(\Omega\) in the viscosity sense.
        \end{itemize}
\end{defn}

In this article, we only require the definition of viscosity subsolutions, so we focus on subsolutions for the remainder of this section; however, one can easily state the analogous statements for supersolutions by replacing \(E\) with its complement. As we mentioned previously, the definition of viscosity subsolution only makes sense if one knows that~\(\curv_{K, E}(x)\) is well-defined in the principal value sense (albeit possibly infinite) whenever~\(E\) is touched at \(x\) by an exterior ball. For a proof of this fact see \cite[Section 4]{cabre_calibrations_2020}, see also \cite[Proposition 5.5]{moy_c1alpha_2025}. In Lemma~\ref{JEUx52ma} below, we restate these results in the form we require for our main results, and give an alternative proof. We also emphasise that, in the context of nonlocal minimal surfaces, there are several different equivalent definitions of viscosity solutions, we refer the interested reader to \cite[Section 2.1]{cozzi_halfspace_inprep}. Note, in the previously mentioned references, sets are identified up to measure zero, which we do not do in this article. 

Before showing this, however, recall that, given a ball \(B\), \(\curv_{K,B}\) is defined in the principal value sense (i.e. that the limit~\eqref{DVf4VFSd} exists). Indeed, \cite[Proposition 6.3]{figalli_isoperimetry_2015} proves that \(\curv_{K,E}\) is defined and is continuous if \(E\) has \(C^{1,1}\) boundary. In the lemma below, we show \(\curv_{K,B}\) is comparable to \(r^dk(r)\), and give the bounds explicitly.

\begin{lemma} \label{XqTvAWsB}
   Let \(K\) satisfy~\eqref{wqeJL6Xs}--\eqref{M4Xlk3oo} and \(B = B_r(x_0)\) be a ball. Then the limit \begin{equation}
        \curv_{K,B}(x)= \lim_{\varepsilon \to 0^+} \int_{\R^d\setminus B_\varepsilon(x)} \tilde \chi_B(y) K(x-y) \,d y \label{hDeOAq9O}
   \end{equation} exists for all \(x\in \partial B\) and  \begin{align}
    \lambda  C_\ast^{-1}r^d k(r) \leq \curv_{ K,B}(x) \leq   \Lambda C_\ast  r^d k(r) \qquad \text{for all }0<r<\bar r /4 \label{9Hg4n0sm}
\end{align} where \(C_\ast \geq 1\) is a constant depending only on the coefficients of \(k\), \(T_k\), and \(d\).
\end{lemma}


\begin{proof}
 We begin by proving that \begin{align}
     \lambda  \curv_{k,B_r}\!(re_1)  \leq    \curv_{K,B_r(x_0)}(x) \le  \Lambda   \curv_{k,B_r}(re_1) \qquad \text{for all }x\in \partial B_r(x_0). \label{gTWrdI5A}
\end{align} To prove~\eqref{gTWrdI5A}, we will first prove the estimate in the particular case \(x_0=re_1\) and \(x=0\). Then we obtain the full claim by an isometric transformation. Indeed, fix \(\varepsilon >0\) arbitrarily small and write \begin{equation*}
    \int_{\R^d\setminus B_\varepsilon} \tilde \chi _{B_r(r e_1)} (y) K(y) \, dy = I_\varepsilon +J
\end{equation*} where \begin{equation*}
    I_\varepsilon = \int_{B_{2r} \setminus B_\varepsilon} \tilde \chi _{B_r(re_1)}(y) K(y) \, dy
\end{equation*} and \begin{equation*}
    J = \int_{ \R^d \setminus B_{2r} } \tilde \chi _{B_r(re_1)} (y) K(y) \, dy .
\end{equation*}For \(I_\varepsilon\), using that \(\tilde \chi_{B_r(re_1)}=\chi_{\R^d\setminus B_r(re_1)}-\chi_{B_r(re_1)}\), we have that \begin{equation*}
    I_\varepsilon = \int_{ (B_{2r} \setminus B_\varepsilon)\setminus B_r(re_1)}  K(y) \, dy -\int_{(B_{2r} \setminus B_\varepsilon)\cap B_r(re_1) }  K(y) \, dy. 
\end{equation*}Making the change of variable \(y \to -y\) in the second integral above and using \eqref{M4Xlk3oo}, we obtain \begin{equation*}\begin{aligned}
    I_\varepsilon &=  \int_{ (B_{2r} \setminus B_\varepsilon)\setminus B_r(re_1)}  K(y) \, dy -\int_{(B_{2r} \setminus B_\varepsilon)\cap B_r(-re_1)}  K(y) \, dy \\
    &= \int_{(B_{2r} \setminus B_\varepsilon)\setminus(  B_r(re_1) \cup B_r(-re_1)) }  K(y) \, dy .
\end{aligned}\end{equation*} Hence, the monotone convergence theorem implies that \(\lim_{\varepsilon \to 0^+} I_\varepsilon\) exists, which further implies that the limit~\eqref{hDeOAq9O} exists. Moreover, by~\eqref{wqeJL6Xs}, we obtain \begin{equation*}
    I_\varepsilon \le  \Lambda \int_{(B_{2r} \setminus B_\varepsilon)\setminus(  B_r(re_1)\cup B_r(-re_1)) } k(\vert y \vert)   \, dy = \Lambda  \int_{B_{2r} \setminus B_\varepsilon} \tilde \chi _{B_r(re_1)}(y)k(\vert y \vert)  \, dy
\end{equation*} Similarly, we also obtain \begin{equation*}
     I_\varepsilon \ge \lambda  \int_{B_{2r} \setminus B_\varepsilon} \tilde \chi _{B_r(re_1)}(y) k(\vert y \vert)  \, dy. 
\end{equation*}The estimate for \(J\) is easy, since \(\tilde \chi _{B_r(re_1)}=1\) in \(\R^d\setminus B_{2r}\), so~\eqref{wqeJL6Xs} immediately implies\begin{equation*}
   \lambda \int_{\R^d\setminus B_{2r}} \tilde \chi _{B_r(re_1)}(y)k(\vert y \vert)   \, dy \le J \le \Lambda \int_{\R^d\setminus B_{2r}} \tilde \chi _{B_r(re_1)}(y)k(\vert y \vert)   \, dy
\end{equation*} Hence, collecting all the estimates above, we obtain \begin{equation*}
    \lambda \int_{\R^d\setminus B_\varepsilon} \tilde \chi _{B_r(re_1)}(y) k(\vert y \vert) \, dy \le \int_{\R^d\setminus B_\varepsilon} \tilde \chi _{B_r(re_1)}(y) K(y) \, dy \le \Lambda \int_{\R^d\setminus B_\varepsilon} \tilde \chi _{B_r(re_1)}(y) k(\vert y \vert) \, dy . 
\end{equation*} Sending \(\varepsilon \to 0^+\) implies \begin{equation*}
    \lambda \curv_{k,B_r(re_1)}\!(0) \le \curv_{K,B_r(re_1)}(0) \le \Lambda \curv_{k,B_r(re_1)}\!(0) .
\end{equation*} Then \begin{align*}
    \curv_{k,B_r(re_1)}\!(0) = \lim_{\varepsilon \to 0^+} \int_{\R^d\setminus B_\varepsilon} \tilde \chi_{B_r(re_1)} (y) k(\vert y\vert ) \,d y = \lim_{\varepsilon \to 0^+} \int_{\R^d\setminus B_\varepsilon(re_1) } \tilde \chi_{B_r} (y) k(\vert y-re_1\vert ) \,d y = \curv_{k,B_r}\!(re_1)
\end{align*} via the change of variable \(y\to y-re_1\), so we obtain~\eqref{gTWrdI5A} when \(x_0=re_1\) and \(x=0\). For the general case, let \(O\) be a \(d\times d\) orthogonal matrix such that \(O(re_1)=x_0-x\). Then, via the change of variable \(y=Oz+x\), it follows that \begin{align*}
    \int_{\R^d \setminus B_{\varepsilon}(x)} \tilde \chi_{B_r(x_0)} (y) K(x-y) \,d y = \int_{\R^d \setminus B_\varepsilon } \tilde \chi_{B_r(re_1) }(z) K(Oz)   \,d z 
\end{align*} for all \(\varepsilon>0\). 
Passing to the limit \(\varepsilon \to 0^+\), we complete the proof of~\eqref{gTWrdI5A}.

Finally,  we prove that \begin{align}
    C^{-1} r^d k(r) \leq \curv_{ k,B_r}\!(re_1) \leq   C r^d k(r) \text{ for all } 0<r<\bar r /4 . \label{IHEKQWFB}
\end{align} where \(C \geq 1\) is a constant depending only on the coefficients of \(k\), \(T_k\), and \(d\). Once we establish~\eqref{IHEKQWFB}, using~\eqref{gTWrdI5A}, we obtain~\eqref{9Hg4n0sm}. As above, we write \begin{align}
    \curv_{ k,B_r}\!(re_1) &= \curv_{ k,B_r(r e_1)}\!(0) \nonumber \\
    &=\lim_{\varepsilon \to 0^+}\int_{B_{2r} \setminus B_\varepsilon} \tilde \chi _{B_r(re_1)}(y) k(\vert y\vert ) \, dy+\int_{\R^d\setminus B_{2r}} \tilde \chi _{B_r(re_1)}(y) k(\vert y\vert ) \, dy \nonumber \\
    &=\int_{B_{2r} \setminus(  B_r(re_1) \cup B_r(-re_1)) }  k(\vert y \vert ) \, dy+\int_{\R^d\setminus B_{2r}} k(\vert y\vert ) \, dy . \label{dml8Ueyh}
\end{align} For \(d=1\), the first integral on the right-hand side of~\eqref{dml8Ueyh} is equal to zero since in this case \(  B_{2r} \setminus(  B_r(re_1) \cup B_r(-re_1))   = \{0\}\) which has zero measure. For \(d\geq 2\), via polar coordinates, we have \begin{align*}
    \int_{B_{2r} \setminus(  B_r(re_1) \cup B_r(-re_1)) }  k(\vert y \vert ) \, dy &=  \int_0^{2r} \vert \partial B_\rho \setminus(  B_r(re_1) \cup B_r(-re_1))\vert \cdot k(\rho) \, d \rho.
\end{align*} Then, rescaling Lemma~\ref{lem:sph-measure} below, we obtain  \begin{align*}
     C_1^{-1} r^{-1} \rho^d \leq \vert \partial B_\rho \setminus(  B_r(re_1) \cup B_r(-re_1))\vert \leq C_1 r^{-1} \rho^d \qquad \text{for all }0<\rho<r 
\end{align*} with \(C_1>0\) depending only on \(d\). Furthermore, we trivially have \(\vert \partial B_\rho \setminus(  B_r(re_1) \cup B_r(-re_1))\vert  \leq \vert \partial B_1\vert \rho^{d-1}\), so \cite[Proposition A.1(c)]{grzywny_potential_2018} and the monotonicity assumption on \(k\) imply  \begin{align}
    \int_{B_{2r}\setminus(  B_r(re_1) \cup B_r(-re_1)) }   k(\vert y \vert ) \, dy 
&\leq C_1  r^{-1} \int_0^r \rho^d k(\rho) \, d \rho+  \vert \partial B_1 \vert \int_r^{2r} \rho^{d-1}k(\rho) \, d \rho  \nonumber  \\
    &\leq C_2 \bigg ( r^d k(r) +k(r) \int_r^{2r} \rho^{d-1}\, d \rho \bigg )  \nonumber \\
    &\leq C_3 r^d k(r) \label{htUvn9zz} 
\end{align} with \(C_2,C_3\) depending only the coefficients of \(k\) and \(d\). Similarly, \begin{align}
    \int_{B_{2r}\setminus(  B_r(re_1) \cup B_r(-re_1)) }  k(\vert y \vert ) \, dy &\geq C_1^{-1} r^{-1} \int_0^r \rho^d k(\rho) \, d \rho \geq C_3^{-1} r^d k(r) \label{Xrbirad7}
\end{align} (after possibly renaming \(C_3\)) by \cite[Proposition A.1]{grzywny_potential_2018}.

For the second integral on the right-hand side of~\eqref{dml8Ueyh}, integration in polar coordinates gives \begin{align*}
    \int_{\R^d\setminus B_{2r}} k(\vert y\vert ) \, dy &= \vert \partial B_1\vert \int_{2r}^{+\infty}\rho^{d-1} k(\rho) \, d\rho \\
    &= \vert \partial B_1\vert \bigg ( \int_{2r}^{ \bar r }\rho^{d-1} k(\rho) \, d\rho + \int_{\bar r}^{+\infty}\rho^{d-1} k(\rho) \, d\rho \bigg ) .
\end{align*} Then, since \(2r<\bar r/2\) by assumption, \cite[Proposition A.1(d)]{grzywny_potential_2018} implies \begin{align}
    C_4^{-1} r^d k(r) \leq  \vert \partial B_1\vert \int_{2r}^{ \bar r }\rho^{d-1} k(\rho) \, d\rho  \leq C_4 r^d k(r) \label{eq:C4}
\end{align} for some \(C_4>0\) depending only on \(d\) and the coefficients of \(k\), and, tautologically, \begin{align*}
    \vert \partial B_1\vert \int_{\bar r}^{+\infty}\rho^{d-1} k(\rho) \, d\rho = C_5 \bar r^d k(\bar r) 
\end{align*} where \begin{align*}
    C_5 = \vert \partial B_1\vert \cdot T_k.
\end{align*} Moreover, using~\eqref{y8vGgWhv}, \(r<\bar r\) and that \(b+d<0\), we have  \begin{align*}
    \bar r^d k(\bar r) \leq B\bar r^{b+d} \cdot r^{-b} k(r) \leq B r^d k(r),
\end{align*} so \begin{align}
    0\leq \vert \partial B_1\vert \int_{\bar r}^{+\infty}\rho^{d-1} k(\rho) \, d\rho\leq B C_5 r^d k(r). \label{eq:C5}
\end{align}Hence, combining~\eqref{eq:C4} and~\eqref{eq:C5} yields  \begin{align}
   C_4^{-1}  r^d k( r)   \leq \int_{\R^d\setminus B_{2r}} k(\vert y\vert ) \, dy \leq C_4 r^d k(r) + B C_5 r^d k(r)  . \label{e5je8J6N}
\end{align} Finally,~\eqref{dml8Ueyh},~\eqref{htUvn9zz},~\eqref{Xrbirad7}, and~\eqref{e5je8J6N} imply~\eqref{IHEKQWFB}.

\end{proof}

In the proof of Lemma~\ref{XqTvAWsB} above, we required the following estimate. Recall \(A \approx_d B\) means that there exists \(C=C(d)>0\) such that \(C^{-1} A \leq B \leq C A \). 

\begin{lemma} \label{lem:sph-measure}
    If \(d\geq 2\) then \begin{align*}
         \vert \partial B_\rho  \setminus(  B_1(e_1) \cup B_1(-e_1))\vert \approx_d  \rho^d \qquad \text{for all }0<\rho<1 .
    \end{align*} 
\end{lemma}

\begin{proof}
    Let \(x'\) denote a point in \(\R^{d-1}\) and \(B_r'\) denote the ball of radius \(r\) centred at \(0\) in \(\R^{d-1}\). Using graphical coordinates, we have that \begin{align*}
        \vert \partial B_\rho  \setminus(  B_1(e_1) \cup B_1(-e_1))\vert  &= 2 \vert (\partial B_\rho)^+ \setminus B_1(e_1) \vert 
        = 2 \rho \int_{B'_\rho  \setminus B'_{ \rho \sqrt{1-\rho^2/4 } }} \frac{dx' }{\sqrt{\rho^2-\vert x'\vert^2}} . 
    \end{align*} Then rescaling and applying polar coordinates yields \begin{align*}
        \vert \partial B_\rho  \setminus(  B_1(e_1) \cup B_1(-e_1))\vert  &= 2 \rho^{d-1} \int_{B'_1  \setminus B'_{ \sqrt{1-\rho^2/4 } }} \frac{dx' }{\sqrt{1-\vert x'\vert^2}}= 2 \vert \partial B'_1\vert  \rho^{d-1} \int^1_{\sqrt{1-\rho^2/4 } } \frac{\tau^{d-2}}{\sqrt{1-\tau^2}}\, d \tau .
    \end{align*} Since \(\rho <1\), we have \(\sqrt{1-\rho^2/4 } \geq \sqrt{3}/2\), so \(\tau^{d-2} (1-\tau^2)^{-1/2} \approx_d (1-\tau)^{-1/2} \) for all \(\rho \in (\sqrt{1-\rho^2/4 },1)\). Hence, \begin{align*}
        \vert \partial B_\rho  \setminus(  B_1(e_1) \cup B_1(-e_1))\vert  &\approx_d \rho^{d-1} \int^1_{\sqrt{1-\rho^2/4 } } \frac{d \tau }{\sqrt{1-\tau}} \\
        & \approx_d \rho^{d-1} \Big (1 - \big ( 1-\rho^2/4 \big)^{1/2}   \Big )^{1/2} \\
        &\approx_d \rho^d 
    \end{align*} as claimed. 
\end{proof}

Now, we can prove that \(\curv_{K, E}(x)\) is well-defined in the principal value sense when \(E\) is touched at \(x\) by an exterior ball. We also give a non-principal value formula for \(\curv_{K, E}\) in this case. 

\begin{lemma} \label{JEUx52ma}
    Let~\(K\) satisfy~\eqref{wqeJL6Xs}-\eqref{M4Xlk3oo}, \(E\subset \R^d\) be a Borel set, \(B\) be a ball of radius \(r>0\) such that \(B\subset \R^d\setminus  E\), and \(x\in \partial E\cap \partial B\). Then the limit \begin{equation}
        \curv_{K,E}(x) = \lim_{\varepsilon \to 0^+} \int_{\R^d \setminus B_\varepsilon(x) } \tilde \chi_E(y) K(x-y) \, dy \label{0ynOArJU}
    \end{equation} exists and is given by \begin{equation}
        \curv_{K,E}(x) = -\curv_{K,B}(x) + 2 \int_{\R^d\setminus (E\cup B)} K(x-y) \, dy \in [-\curv_{K,B}(x),+\infty] \label{tNuyur0n}
    \end{equation}
\end{lemma}

\begin{proof}
    Fix \(\varepsilon>0\) and observe that \begin{equation*}
        \biggl \lvert \int_{\R^d \setminus B_\varepsilon(x) } \tilde \chi_E(y) K(x-y) \, dy \biggr \rvert \leq 2\Lambda \int_{\R^d \setminus B_\varepsilon }k(\vert y \vert) \, d y <+\infty
    \end{equation*} where the second integral is finite by assumptions~\eqref{HKZolO8l} and~\eqref{y8vGgWhv}. Then, it follows that, \begin{equation*}\begin{aligned}
         I_\varepsilon & := \int_{\R^d \setminus B_\varepsilon(x) } \tilde \chi_E(y) K(x-y) \, dy- \int_{\R^d \setminus B_\varepsilon(x) } \tilde \chi_{\R^d\setminus B }(y) K(x-y) \, dy \\ & = 2 \int_{\bigl ( \R^d \setminus (E \cup B) \bigr ) \setminus B_\varepsilon (x) } K(x-y)\, dy .
    \end{aligned}\end{equation*} By the monotone convergence theorem, the integral on the right-hand side above converges to \begin{equation*}
         2 \int_{\R^d \setminus (E \cup B)} K(x-y)\, dy \in [0,+\infty]
    \end{equation*} as \(\varepsilon \to 0^+\). Furthermore, by Lemma~\ref{XqTvAWsB}, \begin{equation*}
        J_\varepsilon := \int_{\R^d \setminus B_\varepsilon(x) } \tilde \chi_{\R^d\setminus B }(y) K(x-y) \, dy = - \int_{\R^d \setminus B_\varepsilon(x) } \tilde \chi_B(y) K(x-y) \, dy \to -\curv_{K,B}(x)
    \end{equation*} as \(\varepsilon \to 0^+\). Then, writing \begin{equation*}
        \int_{\R^d \setminus B_\varepsilon(x) } \tilde \chi_E(y) K(x-y) \, dy = I_\varepsilon + J_\varepsilon ,
    \end{equation*} we see that the limit in~\eqref{0ynOArJU} exists and equals~\eqref{tNuyur0n}.
\end{proof}

\section{Proof of Theorem \texorpdfstring{\ref{Waso0CUq}}{1.2}} \label{BQhFLLRN}

In this section, we collect several lemmata that we require for the proof of Theorem~\ref{Waso0CUq}. Our first lemma is an application of the Hardy--Littlewood maximal inequality which morally states that a set \(E\) of small measure is necessarily sparse at most points \(x\).

\begin{lemma}
\label{lem:hardylittlewood}
Let \(\alpha >0\), \(R>0\), and \(E\subset \R^d \) be a Borel set. There exists \(\vartheta>0\) such that if 
\begin{equation}\label{eq:hardylittlewood:a}
 \lvert E \cap B_R \rvert \le \vartheta \alpha   \lvert B_R \rvert ,
\end{equation}
then the measure of the set
\begin{equation}\label{eq:hardylittlewood}
 D_\alpha = \{x \in B _{R/4} \text{ s.t. } \lvert E \cap B_r(x) \rvert \le \alpha \lvert B_r \rvert \text{ for all } r \in (0, \tfrac{1}{2} R] \}
\end{equation}
is at least \(\frac{1}{2} \lvert B_{R/4} \rvert\). The constant \(\vartheta>0\) depends only on \(d\).
\end{lemma}

\begin{proof}
Let \(f = \ind_{E \cap B_R}\), and let \(M f\) be the corresponding centred Hardy--Littlewood maximal function:
\begin{equation*}\begin{aligned}
 M f(x) & = \sup \biggl\{ \frac{1}{\lvert B_r \rvert} \int_{B_r(x)} \lvert f(y) \rvert dy \text{ s.t. } r > 0 \biggr\} \\
 & = \sup \biggl\{ \frac{\lvert E \cap B_R \cap B_r(x ) \rvert}{\lvert B_ r \rvert} \text{ s.t. } r > 0 \biggr\} .
\end{aligned}\end{equation*}
By the Hardy--Littlewood maximal inequality, there is a constant \(C>0\) depending only on \(d\) such that
\[
 \lvert \{ x \in \R^d \text{ s.t. } M f(x) > \alpha \} \rvert \le C \, \frac{\lVert f \rVert_{L^1(\R^d)} }{\alpha} = C \, \frac{\lvert E \cap B_R\rvert}{\alpha} ,
\] so~\eqref{eq:hardylittlewood:a} implies \begin{equation}
     \lvert \{ x \in \R^d \text{ s.t. } M f(x) > \alpha \} \rvert \le C \vartheta \lvert B_R \rvert = C4^d \vartheta \lvert B_{R/4} \rvert. \label{Xh7nBdPF}
\end{equation}
Next, observe that \begin{equation}
    \{ x \in B_{R/4} \text{ s.t. } M f(x) \le \alpha \} \subset D_\alpha . \label{NWEe2LAv}
\end{equation} Indeed, if \(x\in B_{R/4}\) and \(Mf(x) \leq \alpha\), then for all \(r\in (0, R/2]\), we have \(B_r(x) \subset B_R\), so \begin{equation*}
  \frac{\lvert E  \cap B_r(x ) \rvert}{\lvert B_ r \rvert}= \frac{\lvert E \cap B_R \cap B_r(x ) \rvert}{\lvert B_ r \rvert} \leq M f(x) \leq \alpha .
\end{equation*}
Hence,~\eqref{Xh7nBdPF} and~\eqref{NWEe2LAv} imply \begin{equation*}\begin{aligned}
    \lvert D_\alpha  \rvert & \ge \lvert \{ x \in B_{R/4} \text{ s.t. } M f(x) \le \alpha \} \rvert  \\
    &\geq \lvert B_{R/4} \rvert - \lvert \{ x \in \R^d : M f(x) > \alpha \} \rvert \\
    &\geq  (1-4^d C \vartheta) \lvert B_{R/4}\rvert .
\end{aligned}\end{equation*} Choosing \(\vartheta=\frac{1}{2} (4^d C)^{-1}\), we obtain the desired inequality.
\end{proof}

Next, we present a simple rearrangement inequality well-known to experts. The result in the full space is given in \cite[Lemma 6.1]{di_nezza_hitchhikers_2012} or on a set of finite \((d-1)\)-Hausdorff measure in \cite[Lemma 3.2]{cabre_fractional_2023}. We give the proof here for completeness, but it is identical to the one given in \cite[Lemma 3.2]{cabre_fractional_2023} with the Lebesgue measure replacing the Hausdorff measure.

\begin{lemma} \label{UMd4Dp6e}
Let \(\Omega \subset \R^d\) be measurable with diameter less than $\bar r$, \(x\in \overline \Omega\), and \(E\) be a Borel set. If \(\rho>0\) is such that \begin{equation}
    \lvert \Omega \cap B_\rho(x) \rvert = \lvert  \Omega\cap E \rvert \label{ATiwjvpo}
\end{equation} then \begin{equation*}
    \int_{ \Omega  \setminus  E } k(\vert x-y\vert) \, dy  \geq \int_{ \Omega  \setminus  B_\rho (x) } k(\vert x-y\vert) \, dy .
\end{equation*}
\end{lemma}

\begin{proof}
    The result holds trivially if \(\rho \ge \bar r\), so we only consider \(\rho < \bar r\). By~\eqref{ATiwjvpo}, \begin{equation*}\begin{aligned}
    \lvert \Omega \cap  B_\rho (x) \cap E\rvert+ \lvert (\Omega \cap  B_\rho (x) )\setminus  E\rvert &=     \lvert \Omega \cap  B_\rho (x) \rvert\\
    &=  \lvert  \Omega\cap E \rvert \\
    &=  \lvert  \Omega\cap E \cap  B_\rho (x) \rvert + \lvert  (\Omega\cap E)\setminus   B_\rho (x) \rvert ,
    \end{aligned}\end{equation*} so \begin{equation*}
        \lvert (\Omega \cap  B_\rho (x) )\setminus  E\rvert =   \lvert  (\Omega\cap E)\setminus   B_\rho (x) \rvert.
    \end{equation*} Then, since \(k\) is non-increasing on \((0, \bar r \textbf{})\) and \(\lvert x - y \rvert < \bar r\) for \(x, y \in \Omega\), it follows that \begin{equation*}\begin{aligned}
        \int_{ \Omega  \setminus  E } k(\vert x-y\vert) \, dy &=  \int_{ (\Omega  \setminus  E) \cap B_\rho (x)  }  k(\vert x-y\vert) \, dy +  \int_{ \Omega  \setminus ( E\cup B_\rho(x) )  }  k(\vert x-y\vert) \, dy \\
        &\geq k(\rho) \lvert  (\Omega \cap B_\rho(x) ) \setminus  E \rvert +  \int_{ \Omega  \setminus ( E\cup B_\rho(x) )  } k(\vert x-y\vert)\, dy  \\
        &=  k(\rho)\lvert  (\Omega\cap E)\setminus   B_\rho (x) \rvert +  \int_{ \Omega  \setminus ( E\cup B_\rho(x) )  } k(\vert x-y\vert)\, dy  \\
        &\geq \int_{ (\Omega\cap E)\setminus   B_\rho (x)  }k(\vert x-y\vert)\, dy +  \int_{ \Omega  \setminus ( E\cup B_\rho(x) )  }k(\vert x-y\vert) \, dy  \\
        &= \int_{\Omega \setminus B_\rho(x)} k(\vert x-y\vert)\, dy ,
    \end{aligned}\end{equation*}as required. 
\end{proof}

In our final lemma of this section, we prove that if \(E\) has small measure around a point in its boundary with an exterior touching ball then its nonlocal curvature must be large. 

\begin{lemma} \label{RchNzy7Y}
     Let \(M\geq 0\), \(K\) satisfy~\eqref{wqeJL6Xs}--\eqref{M4Xlk3oo}, and \(E\subset \R^d \) be a Borel set. Suppose that \(r \in (0, \frac{1}{4}\bar r)\), \(B = B_r(x_0)\subset \R^d\setminus E\) and \(x\in \partial E \cap \partial B\). Then, there exists \(\beta>0\) such that if \begin{equation}
         \lvert E \cap B_{2r}(x_0) \rvert \le \beta \lvert B_{2r}(x_0) \rvert \label{dh3rRc1i}
     \end{equation} then \(\curv_{K,E}(x) \ge (M + 1) r^d k(r)\). The constant \(\beta>0\) depends only on \(d\), \(\lambda\), \(\Lambda\), \(M\), the coefficients of \(k\), and \(T_k\). 
\end{lemma}

\begin{proof}
  It follows from Lemma~\ref{JEUx52ma} that \begin{equation*}
        \curv_{K,E} (x) =-\curv_{K,  B}  (x) + 2 \int_{\R^d \setminus  ( E\cup B) } K(x-y) \, dy 
    \end{equation*} By Lemma~\ref{XqTvAWsB}, it follows that \begin{equation*}
        \curv_{K,E} (x) \ge -\Lambda C_\ast r^dk(r)  + 2 \int_{\R^d \setminus  ( E\cup B  ) } K(x-y) \, dy 
    \end{equation*} with \(C_\ast>0\) as in Lemma~\ref{XqTvAWsB}. We claim that, assuming~\eqref{dh3rRc1i} for some sufficiently small \(\beta>0\), implies \begin{equation*}
        2\int_{\R^d \setminus  ( E\cup B ) } K(x-y) \, dy \ge (M + 1 + \Lambda C_\ast ) r^d k(r) .
    \end{equation*} Once we have established the claim, we obtain \begin{equation*}
         \curv_{K,E} (x) \ge (M+1) r^d k(r) ,
    \end{equation*} as desired.

To prove the claim, we use that 
    \begin{equation*}\begin{aligned}
        I:=2 \int_{\R^d \setminus  ( E\cup B ) } K(x-y) \, dy &\ge 2 \lambda \int_{\R^d \setminus  ( E\cup B) }   k(\vert x-y \vert) \, dy\\
        &\ge 2 \lambda \int_{ (B_r(x) \setminus B  ) \setminus  E }   k(\vert x-y \vert) \, dy
    \end{aligned}\end{equation*} Next, if \(\bigl\lvert \bigl ( B_r(x) \setminus B \bigr ) \cap  E \bigr\rvert >0\) and \(\beta < 2^{-d - 1}\), then there exists \(\rho \in (0,r)\) such that \begin{equation}
        \bigl\lvert \bigl ( B_r(x) \setminus B \bigr ) \cap B_\rho(x) \bigr\rvert = \bigl\lvert \bigl ( B_r(x) \setminus B \bigr ) \cap  E \bigr\rvert , \label{dNIUKzja}
    \end{equation} so, by Lemma~\ref{UMd4Dp6e} with \(\Omega = B_r(x) \setminus B\) (note that \(\Omega\) has diameter less than \(\bar r \)), we have  \begin{equation*}
         I\geq 2 \lambda \int_{  ( B_{r}(x) \setminus B  ) \setminus  B_\rho (x) }  k(\vert x-y \vert) \, dy  
    \end{equation*} with \(C_2\) depending only on \(d\). Moreover, since \( B_r(x)\setminus B \) contains a half ball of radius \(r\) centred at \(x\), it further follows that \begin{equation}
       I \geq \lambda \int_{  B_r(x)  \setminus  B_\rho(x) }  k(\vert x-y \vert) \, dy= C_2 \lambda \int_\rho^r s^{d - 1} k(s) ds. \label{qtmEqgRU}
    \end{equation}  If \(\bigl\lvert \bigl ( B_r(x) \setminus B \bigr ) \cap  E \bigr\rvert =0\) then, using again that \( B_r(x)\setminus B \) contains a half ball of radius \(r\) centred at \(x\), we have \begin{align*}
        I \geq  2 \lambda \int_{ B_r(x) \setminus B   }   k(\vert x-y \vert) \, dy \geq \lambda \vert \partial B_1 \vert \int_0^r s^{d-1}k(s) \, ds=+\infty.
    \end{align*} Next, since \(E\cap B = \varnothing\) and \(B_r(x) \subset B_{2r}(x_0)\), it follows from~\eqref{dNIUKzja} and~\eqref{dh3rRc1i} that  \begin{equation*}\begin{aligned}
        \lvert B_\rho(x) \setminus B \rvert & = \bigl\lvert \bigl ( B_r(x) \setminus B \bigr ) \cap B_\rho(x) \bigr\rvert \\ & = \bigl\lvert \bigl ( B_r(x) \setminus B \bigr ) \cap  E \bigr\rvert \\ & = \lvert B_r (x)  \cap  E \rvert \\ & \le \lvert B_{2r} (x_0)  \cap  E \rvert \\ & \le \beta \lvert B_{2r}(x_0) \rvert \\ & = \beta \lvert B_1 \rvert (2r)^d .
    \end{aligned}\end{equation*} On the other hand, \(B_\rho(x) \setminus B\) contains a half ball of radius \(\rho\) centred at \(x\), and so \begin{equation*}
        \lvert B_\rho(x) \setminus B \rvert \ge \tfrac{1}{2} \lvert B_\rho(x) \rvert = \tfrac{1}{2} \rho^d \lvert B_1 \rvert .
    \end{equation*} This leads to \(\tfrac{1}{2} \rho^d \le \beta (2r)^d\), so \(\rho \le 2r (2\beta)^{1/d}\). Suppose that \(\beta\) is small enough, so that \(4 (2 \beta)^{1 / d} \le 1\). Then \(\rho \le \tfrac{1}{2} r\), and so~\eqref{qtmEqgRU} and~\eqref{BcfVd3vx} yield \begin{equation*}
        I \ge C_2 \lambda \int_\rho^{2 \rho} s^{d - 1} k(s) ds \ge \frac{C_2 \lambda}{d} \rho^d k(2 \rho) .
    \end{equation*} By~\eqref{y8vGgWhv}, for \(\beta\) sufficiently small we obtain \((r / \rho)^{-b - d} \ge 2^{-b} B (M + 1 + \Lambda C_\ast ) d / (C_2 \lambda)\), and hence \begin{equation*}
        I \geq \frac{C_2 \lambda}{B d} \rho^d \frac{(2 \rho)^b}{r^b} k(r) \geq (M + 1 + \Lambda C_\ast) r^d k(r) ,
    \end{equation*} which is exactly the claim.
\end{proof}

We are now in a position to give the proof of Theorem~\ref{Waso0CUq}.

\begin{proof}[Proof of Theorem~\ref{Waso0CUq}]
 For the sake of contradiction, assume that \begin{equation*}
     \lvert E \cap B_R\rvert < \delta \lvert B_R \rvert 
\end{equation*} for some small \(\delta>0\), and fix \(\alpha \in (0,1)\), both to be chosen later. Let \begin{equation*}
    D_\alpha = \{x \in B _{R/4} \text{ s.t. } \lvert E \cap B_r(x) \rvert \le \alpha \lvert B_r \rvert \text{ for all } r \in (0, \tfrac{1}{2} R] \}
\end{equation*} as in~\eqref{eq:hardylittlewood}. By Lemma~\ref{lem:hardylittlewood}, there exists \(\vartheta > 0\) depending only on \(d\) such that if \(\delta \le \vartheta \alpha\), then \begin{equation}
    \lvert D_\alpha \rvert \ge \frac 12 \lvert B_{R/4} \rvert >0. \label{0cJF4yr9}
\end{equation} We claim that if \(\alpha\) is less than some constant depending only on \(d\), \(\lambda\), \(\Lambda\), \(M\), the coefficients of \(k\), and \(T_k\) then \begin{equation}
    D_\alpha \subset \overline E . \label{3hl93LBN}
\end{equation} To prove the claim, suppose there exists \(x_0\in (\R^d\setminus \overline E) \cap D_\alpha\) and let \(r>0\) be the supremum over radii \(\rho>0\) such that \(B_\rho(x_0) \subset \R^d\setminus E\). Then, there exists \(x \in \partial E \cap \partial B_r(x_0)\). Furthermore, since \(0 \in \partial E \cap B_{R/4}(x_0)\), we have that \(r \le R/4\) and \(x \in B_{R/2}\). Since \(x_0\in D_\alpha\) and \(r\in (0,R/4]\), we have that \begin{equation*}
        \lvert E \cap B_{2r}(x_0) \rvert \le \alpha \lvert B_{2r}(x_0) \rvert.
\end{equation*} Hence, by Lemma~\ref{RchNzy7Y}, there exists a constant \(\beta>0\) which depends only on \(d\), \(\lambda\), \(\Lambda\), \(M\), the coefficients of \(k\), and \(T_k\) such that if \(\alpha\le\beta\), then \(\curv_{K,E}(x) \geq (B M+1) r^d k(r)\). Finally, \(\curv_{K,E}\le M R^d k(R)\) in \(B_{R/2}\) in the viscosity sense and \(r<R\), so \begin{align*}
    \curv_{K,E}\le M R^d k(R) \leq BM R^{b+d}r^{-b} k(r) \leq BM r^d k(r) \qquad \text{in }B_{R/2}
\end{align*} contradicting the previous estimate. 
     
We can now specify the values of \(\alpha\) and \(\delta\): we choose \(\alpha = \beta\) and \(\delta = \vartheta \alpha\). To complete the proof, observe that \(\alpha < 1\), and so \(D_\alpha\) contains no interior point of \(E\). It follows from~\eqref{3hl93LBN} that \(D_\alpha \subset \partial E\), and~\eqref{0cJF4yr9} shows that the boundary of \(E\) contains a subset \(D_\alpha\) of positive measure. This contradicts the assumption that \(\partial E\) has zero Lebesgue measure. 
\end{proof}

For completeness, we give the proof of Corollary~\ref{ED1qK30X} which is a simple application of Theorem~\ref{Waso0CUq}.

\begin{proof}[Proof of Corollary~\ref{ED1qK30X}]
Since \(B_{r/2}\subset B_R\), we have that \(E\) satisfies \(\curv_{K,E} \leq M R^d k(R) \) in \(B_{r/2}\). Then Assumption~\ref{y8vGgWhv} implies \begin{align*}
    R^d k(R) \leq 2^b B R^{b+d}r^{-b} k(r/2) \leq 2^{b-a} A^{-1} B R^{b+d}r^{-b} k(r)\leq 2^{-d-a} A^{-1} B r^d k(r) 
\end{align*} since \(b+d<0\), \(r< 2R\), \(r<\bar r\). Hence \(\curv_{K,E} \leq2^{-d-a} A^{-1} B M r^d k(r) \) in \(B_{r/2}\), so the result follows from Theorem~\ref{Waso0CUq} with \(M\) replaced with \(2^{-d-a} A^{-1} BM\) and \(R\) replaced with~\(r\).
\end{proof}

We also give the proof of Corollary~\ref{IP8ScC2r}: 

\begin{proof}[Proof of Corollary~\ref{IP8ScC2r}]
    Since \(\curv_{K,E} = 0\) in \(\R^d\) in the viscosity sense, we also have \(\curv_{K,\R^d\setminus E} = 0\) in \(\R^d\) in the viscosity sense, so Corollary~\ref{k9pGd2yU} applied to both \(E\) and \(\R^d\setminus E\) yields \begin{equation*}
        \min \{ \lvert E \cap B_r\rvert , \lvert B_r\setminus E \rvert \} \geq C_1 \lvert B_r \rvert \qquad \text{for all }r>0,
    \end{equation*} where \(C_1 > 0\) is a constant that depends only on \(d\), \(s\), \(\lambda\), and \(\Lambda\). For \(\alpha=1\), we conclude the result using the relative isoperimetric inequality~\eqref{yzR232ke} see \cite[Proposition 12.37]{maggi_sets_2012}. For \(0<\alpha<1\), in the proof of \cite[Lemma 3.1]{bellido_existence_2014}, they establish \begin{equation*}
        \int_{B_r} \biggl \lvert u(x) - \frac{1}{\lvert B_r \rvert} \int_{B_r}u (y) \, d y \biggr \rvert \, dx \leq C_2 r^\alpha \int_{B_r}\int_{B_r} \frac{\lvert u(x)-u(y)\rvert }{\lvert x-y\rvert^{d+\alpha}} \, dy \, dx
    \end{equation*} for all \(u\in L^1_{\mathrm {loc}}(\R^d)\) such that right-hand side above is finite, with a constant \(C_2 > 0\) that depends only on \(d\) and \(\alpha\). In particular, if \(u=\chi_E\) then \begin{equation*}
        \int_{B_r}\int_{B_r} \frac{\lvert \chi_E(x)-\chi_E(y)\rvert }{\lvert x-y\rvert^{d+\alpha}} \, dy \, dx =\frac 12 \int_{B_r}\int_{B_r} \frac{\lvert \tilde \chi_E(x)-\tilde \chi_E(y)\rvert }{\lvert x-y\rvert^{d+\alpha}} \, dy \, dx \leq 2\operatorname{Per}_\alpha(E;B_r)
    \end{equation*} by~\eqref{rKeVDPTi} and \begin{equation*}
         \int_{B_r} \biggl \lvert \chi_E(x) - \frac{1}{\lvert B_r \rvert} \int_{B_r}\chi_E (y) \, d y \biggr \rvert \, dx = \frac{2 \lvert E \cap B_r \rvert \cdot \lvert B_r\setminus E \rvert }{\lvert B_r \rvert }\geq \min \{ \lvert E \cap B_r\rvert , \lvert B_r\setminus E \rvert \} ,
    \end{equation*} so \begin{equation*}
        \operatorname{Per}_\alpha(E;B_r) \geq (2C_2)^{-1} r^{-\alpha} \min \{ \lvert E \cap B_r\rvert , \lvert B_r\setminus E \rvert \}\geq C_1 (2C_2)^{-1} \lvert B_1 \rvert r^{d-\alpha}= C_3 \operatorname{Per}_\alpha (B_r),
    \end{equation*}
    where \(C_3 > 0\) depends only on \(d\), \(\alpha\), \(s\), \(\lambda\), and \(\Lambda\).
\end{proof}

\section{Proof of Theorem~\texorpdfstring{\ref{thm:main}}{1.6}} \label{6yuqjipW}

In this section, we give the proof of Theorem~\ref{thm:main}. It is very similar to the proof of Theorem~\ref{Waso0CUq}, but requires two extra lemmas which strengthen Lemma~\ref{RchNzy7Y}. They show that if a set \(E\) is sparse at \(x\), then the nonlocal mean curvature of \(E\) is infinite at \(x\). The first lemma is as follows:

\begin{lemma}
\label{lem:shell}
Let \(K\) satisfy~\eqref{wqeJL6Xs}--\eqref{M4Xlk3oo}, \(E\subset \R^d \) be a Borel set, and \(x\in \partial E\). Then, there exist \(C, \gamma>0\) such that if 
\begin{equation}\label{eq:shell:a}
 \lvert E \cap B_r(x) \rvert \le \gamma \lvert B_r \rvert ,
\end{equation}
with \(0<r\leq \bar r\) then \begin{equation*}
    \int_{B_r(x) \setminus B_{ r/2}(x)} \tilde \ind_E(y) K(x-y)  \, dy \ge  C r^d k(r) \, .
\end{equation*}The constants \(\gamma\) and \(C\) depend only on \(d\), \(\lambda\), \(\Lambda\), and the coefficients of \(k\).
\end{lemma}

\begin{proof}
Since \(\tilde \ind_E = 1 - 2 \ind_E\), we have \begin{equation*}
    \int_{B_r(x)\setminus B_{ r/2}(x) } \tilde  \chi_E (y) K(x-y) \, dy  
   = \int_{B_r(x)\setminus B_{r/2}(x)}  K(x-y) \, dy-2\int_{  \bigl(  B_r(x)\setminus B_{ r/2 }(x)  \bigr) \cap E}  K(x-y) \, dy 
\end{equation*}
We can estimate the former integral on the right-hand side by: \begin{equation*}
   \int_{B_r(x)\setminus B_{r/2}(x)}  K(x-y) \, dy \geq \lambda \int_{B_r(x)\setminus B_{r/2}(x)}   k(\vert x-y\vert)  \, dy = \lambda \lvert \partial B_1\rvert\int^r_{r/2} t^{d-1}k(t)\, dt \geq C_1 r^dk(r)  
\end{equation*} using \cite[Proposition A.1(d)]{grzywny_potential_2018} where \(C_1 > 0\) depends only on \(d\), \(\lambda\), and coefficients of \(k\). In the latter one, we estimate the integrand by \(k(r/2)\) and use the assumption~\eqref{eq:shell:a} to obtain: \begin{equation*}
    \int_{  \bigl(  B_r(x)\setminus B_{ r/2}(x)  \bigr) \cap E}  K(x-y) \, dy \le  \Lambda \int_{  \bigl(  B_r(x)\setminus B_{ r/2}(x)  \bigr) \cap E} k (\vert x-y\vert)  \, dy 
    \le  \Lambda k (r/2)\lvert E \cap B_r( x) \rvert  
    \le   C_2 \gamma  r^dk(r) ,
\end{equation*} with \(C_2 > 0\) depending only on \(d\), \(\Lambda\), and the coefficients of \(k\). It follows that \begin{equation*}
    \int_{B_r(x)\setminus B_{r/2}(x) } \tilde \chi_E (y) K(x-y) \, dy  \geq   ( C_1- 2C_2\gamma ) r^d k(r) 
\end{equation*} Choosing \(\gamma =C_1/(4C_2)\), we obtain \begin{equation*}
    \int_{B_r(x)\setminus B_{r/2}(x) } \tilde \chi_E (y) K(x-y) \, dy  \geq \frac{C_1 }2 r^d k(r)  ,
\end{equation*}
as desired.
\end{proof}

The second lemma is as follows:

\begin{lemma}
\label{lem:negative}
Let \(K\) satisfy~\eqref{wqeJL6Xs}--\eqref{M4Xlk3oo}, \(E \subset \R^d\) be a Borel set, \(x \in \partial E\), and \(r_0 \in (0,\bar r )\). Then there exists \(\gamma>0\) such that if \begin{equation*}
    \lvert E \cap B_r(x) \rvert \le \gamma \lvert B_r \rvert \qquad \text{for all }r\in (0,r_0 ]
\end{equation*} then \(\curv_{K,E}(x)=+\infty\). The constant \(\gamma\) depends only on \(d\), \(\lambda\), \(\Lambda\), the coefficients of \(k\), and \(T_k\).
\end{lemma}

\begin{proof}
  Let \(C, \gamma>0\) be as in Lemma~\ref{lem:shell}. For every \(r \in (0, \tfrac{1}{2} r_0]\) there is an integer \(m \ge 1\) such that \(r \in (2^{-m - 1} r_0, 2^{-m} r_0]\). If we denote \(\tilde B_j = B_{2^j r}(x)\), then \begin{equation*}
        \int_{\R^d \setminus B_r(x )}\tilde \chi_E(y) K(x-y) \, dy 
        =\biggl(\int_{\R^d \setminus \tilde B_m} + \sum_{j = 1}^m \int_{\tilde B_j \setminus \tilde B_{j-1}} \biggr) \tilde \ind_E (y)K(x-y)\, dy .
    \end{equation*} Since \(2^j r \in (0, r_0 ]\subset (0,\bar r)\) for all \(j = 1, \ldots, m\), Lemma~\ref{lem:shell} applies to the latter integrals on the right-hand side: \begin{equation*}
        \int_{\tilde B_j \setminus \tilde B_{j-1}} \tilde \ind_E (y)K(x-y)\, dy \ge C  ( 2^j r )^d k(2^jr). 
    \end{equation*} 
    In the former one, we estimate \(\tilde \ind_E(y) \ge -1\) to obtain \begin{align*}
          \int_{\R^d \setminus \tilde B_m}  \tilde \ind_E(y) K(x-y) \, dy  &\ge - \Lambda \int_{\R^d \setminus \tilde B_m}  k(\vert x-y\vert)  \, dy \\
          &= -\Lambda \vert \partial B_1 \vert \int_{2^mr}^{+\infty} t^{d-1}k(t) \, dt \\
          &\ge -\Lambda \vert \partial B_1 \vert \int_{r_0/2}^{+\infty} t^{d-1}k(t) \, dt \\
          &=: -C_0
    \end{align*} where \(C_0 > 0\) is independent of \(r\). It follows that \begin{equation}
        \int_{\R^d \setminus B_r(x )}\tilde \chi_E(y) K(x-y) \, dy  \geq - C_0  + C \sum_{j = 1}^m ( 2^j r )^d k(2^jr) \geq - C_0 + C( 2 r )^d k(2 r) ; \label{eq:estimate}
    \end{equation} in the last inequality, we discarded all terms except \(j = 1\). Finally, we have that \(r^d k(r) \to +\infty\) as \(r \to 0^+\) since \begin{align*}
        r^d k(r) \geq \frac{ r_0^{-b} k(r_0)}{B} r^{b+d}\to +\infty \text{ as } r\to 0^+
    \end{align*} using that \(b+d<0\). Hence, sending \(r \to 0^+\) in~\eqref{eq:estimate}, we obtain \begin{equation*}
        \curv_{K,E}(x)=+\infty ,
    \end{equation*} as desired.
\end{proof}

Finally, we give the proof of Theorem~\ref{thm:main}.

\begin{proof}[Proof of Theorem~\ref{thm:main}] We argue exactly as in the proof of Theorem~\ref{Waso0CUq}, but rather than \(\alpha = \beta\), we choose \(\alpha = \min \{\beta, \gamma\}\), with \(\beta\) defined in the proof of Theorem~\ref{Waso0CUq} (and coming from Lemma~\ref{RchNzy7Y}) and \(\gamma\) the constant of Lemma~\ref{lem:negative}. The set \(D_\alpha\) is as in~\eqref{eq:hardylittlewood}, and the proof of Theorem~\ref{Waso0CUq} shows that \(D_\alpha \subset \partial E\). But then, for any \(x\in D_\alpha\), the assumptions of Lemma~\ref{lem:negative} are satisfied with \(r_0=\frac 12 R\), so \(\curv_{K,E}(x)=+\infty\). Hence, \begin{equation*}
    D_\alpha \subset \{ x\in \partial E \cap B_{R/4}\text{ s.t. } \curv_{K,E}(x)=+\infty \} 
\end{equation*} which completes the proof using that \( D_\alpha \) has positive Lebesgue measure. 
\end{proof}

\subsection*{Acknowledgements}

The authors would like to thank Serena Dipierro, Matteo Cozzi, and Enrico Valdinoci for their comments on a preliminary version of this article. Part of this work was carried out while both authors visited the Erwin Schrödinger International Institute for Mathematics and Physics. We thank the ESI for their hospitality.

\appendix

\printbibliography
\vfill

\end{document}